\documentclass[10pt]{article}
\usepackage{amsmath}
\usepackage{amssymb}
\usepackage{amsthm}
\usepackage{mathrsfs}

\begin{document}
\title{Intrinsic square functions on the weighted Morrey spaces}
\author{Hua Wang \footnote{E-mail address: wanghua@pku.edu.cn.
Supported by National Natural
Science Foundation of China under Grant \#10871173 and \#10931001.}\\
\footnotesize{Department of Mathematics, Zhejiang University, Hangzhou 310027, China}}
\date{}
\maketitle

\begin{abstract}
In this paper, we will study the boundedness properties of intrinsic square functions including the Lusin area integral, Littlewood-Paley $g$-function and $g^*_\lambda$-function on the weighted Morrey spaces $L^{p,\kappa}(w)$ for $1<p<\infty$ and $0<\kappa<1$. The corresponding commutators generated by $BMO(\mathbb R^n)$ functions and intrinsic square functions are also discussed.
\end{abstract}

\section{Introduction and main results}

Let ${\mathbb R}^{n+1}_+=\mathbb R^n\times(0,\infty)$ and $\varphi_t(x)=t^{-n}\varphi(x/t)$. The classical square function (Lusin area integral) is a familiar object. If $u(x,t)=P_t*f(x)$ is the Poisson integral of $f$, where $P_t(x)=c_n\frac{t}{(t^2+|x|^2)^{{(n+1)}/2}}$ denotes the Poisson kernel in ${\mathbb R}^{n+1}_+$. Then we define the classical square function(Lusin area integral) $S(f)$ by
\begin{equation*}
S(f)(x)=\bigg(\iint_{\Gamma(x)}\big|\nabla u(y,t)\big|^2t^{1-n}\,dydt\bigg)^{1/2},
\end{equation*}
where $\Gamma(x)$ denotes the usual cone of aperture one:
\begin{equation*}
\Gamma(x)=\big\{(y,t)\in{\mathbb R}^{n+1}_+:|x-y|<t\big\}
\end{equation*}
and
\begin{equation*}
\big|\nabla u(y,t)\big|=\left|\frac{\partial u}{\partial t}\right|^2+\sum_{j=1}^n\left|\frac{\partial u}{\partial y_j}\right|^2.
\end{equation*}
We can similarly define a cone of aperture $\beta$ for any $\beta>0$:
\begin{equation*}
\Gamma_\beta(x)=\big\{(y,t)\in{\mathbb R}^{n+1}_+:|x-y|<\beta t\big\},
\end{equation*}
and corresponding square function
\begin{equation*}
S_\beta(f)(x)=\bigg(\iint_{\Gamma_\beta(x)}\big|\nabla u(y,t)\big|^2t^{1-n}\,dydt\bigg)^{1/2}.
\end{equation*}
The Littlewood-Paley $g$-function (could be viewed as a ``zero-aperture" version of $S(f)$) and the $g^*_\lambda$-function (could be viewed as an ``infinite aperture" version of $S(f)$) are defined respectively by
\begin{equation*}
g(f)(x)=\bigg(\int_0^\infty\big|\nabla u(x,t)\big|^2 t\,dt\bigg)^{1/2}
\end{equation*}
and
\begin{equation*}
g^*_\lambda(f)(x)=\left(\iint_{{\mathbb R}^{n+1}_+}\bigg(\frac t{t+|x-y|}\bigg)^{\lambda n}\big|\nabla u(y,t)\big|^2 t^{1-n}\,dydt\right)^{1/2}.
\end{equation*}

The modern (real-variable) variant of $S_\beta(f)$ can be defined in the following way. Let $\psi\in C^\infty(\mathbb R^n)$ be real, radial, have support contained in $\{x:|x|\le1\}$, and $\int_{\mathbb R^n}\psi(x)\,dx=0$. The continuous square function $S_{\psi,\beta}(f)$ is defined by
\begin{equation*}
S_{\psi,\beta}(f)(x)=\bigg(\iint_{\Gamma_\beta(x)}\big|f*\psi_t(y)\big|^2\frac{dydt}{t^{n+1}}\bigg)^{1/2}.
\end{equation*}

In 2007, Wilson \cite{wilson1} introduced a new square function called intrinsic square function which is universal in a sense (see also \cite{wilson2}). This function is independent of any particular kernel $\psi$, and it dominates pointwise all the above defined square functions. On the other hand, it is not essentially larger than any particular $S_{\psi,\beta}(f)$. For $0<\alpha\le1$, let ${\mathcal C}_\alpha$ be the family of functions $\varphi$ defined on $\mathbb R^n$ such that $\varphi$ has support containing in $\{x\in\mathbb R^n: |x|\le1\}$, $\int_{\mathbb R^n}\varphi(x)\,dx=0$, and for all $x, x'\in \mathbb R^n$,
\begin{equation*}
|\varphi(x)-\varphi(x')|\le|x-x'|^\alpha.
\end{equation*}
For $(y,t)\in {\mathbb R}^{n+1}_+$ and $f\in L^1_{{loc}}(\mathbb R^n)$, we set
\begin{equation*}
A_\alpha(f)(y,t)=\sup_{\varphi\in{\mathcal C}_\alpha}\big|f*\varphi_t(y)\big|=\sup_{\varphi\in{\mathcal C}_\alpha}\bigg|\int_{\mathbb R^n}\varphi_t(y-z)f(z)\,dz\bigg|.
\end{equation*}
Then we define the intrinsic square function of $f$ (of order $\alpha$) by the formula
\begin{equation*}
\mathcal S_\alpha(f)(x)=\left(\iint_{\Gamma(x)}\Big(A_\alpha(f)(y,t)\Big)^2\frac{dydt}{t^{n+1}}\right)^{1/2}.
\end{equation*}
We can also define varying-aperture versions of $\mathcal S_\alpha(f)$ by the formula
\begin{equation*}
\mathcal S_{\alpha,\beta}(f)(x)=\left(\iint_{\Gamma_\beta(x)}\Big(A_\alpha(f)(y,t)\Big)^2\frac{dydt}{t^{n+1}}\right)^{1/2}.
\end{equation*}
The intrinsic Littlewood-Paley $g$-function and the intrinsic $g^*_\lambda$-function will be defined respectively by
\begin{equation*}
g_\alpha(f)(x)=\left(\int_0^\infty\Big(A_\alpha(f)(x,t)\Big)^2\frac{dt}{t}\right)^{1/2}
\end{equation*}
and
\begin{equation*}
g^*_{\lambda,\alpha}(f)(x)=\left(\iint_{{\mathbb R}^{n+1}_+}\left(\frac t{t+|x-y|}\right)^{\lambda n}\Big(A_\alpha(f)(y,t)\Big)^2\frac{dydt}{t^{n+1}}\right)^{1/2}.
\end{equation*}

In \cite{wilson2}, Wilson proved the following result.

\newtheorem*{thma}{Theorem A}
\begin{thma}
Let $0<\alpha\le1$, $1<p<\infty$ and $w\in A_p(\mbox{Muckenhoupt weight class})$. Then there exists a constant $C>0$ independent of $f$ such that
\begin{equation*}
\|\mathcal S_\alpha(f)\|_{L^p_w}\le C \|f\|_{L^p_w}.
\end{equation*}
\end{thma}

Moreover, in \cite{lerner}, Lerner showed sharp $L^p_w$ norm inequalities for the intrinsic square functions in terms of the $A_p$ characteristic constant of $w$ for all $1<p<\infty$. As for the boundedness of intrinsic square functions on the weighted Hardy spaces $H^p_w(\mathbb R^n)$ for $n/{(n+\alpha)}\le p\le1$, we refer the readers to \cite{huang}, \cite{wang3} and \cite{wang4}.

Let $b$ be a locally integrable function on $\mathbb R^n$, in this paper, we will also consider the commutators generated by $b$ and intrinsic square functions, which are defined respectively by the following expressions
\begin{equation*}
\big[b,\mathcal S_\alpha\big](f)(x)=\left(\iint_{\Gamma(x)}\sup_{\varphi\in{\mathcal C}_\alpha}\bigg|\int_{\mathbb R^n}\big[b(x)-b(z)\big]\varphi_t(y-z)f(z)\,dz\bigg|^2\frac{dydt}{t^{n+1}}\right)^{1/2},
\end{equation*}
\begin{equation*}
\big[b,g_\alpha\big](f)(x)=\left(\int_0^\infty\sup_{\varphi\in{\mathcal C}_\alpha}\bigg|\int_{\mathbb R^n}\big[b(x)-b(y)\big]\varphi_t(x-y)f(y)\,dy\bigg|^2\frac{dt}{t}\right)^{1/2},
\end{equation*}
and
\begin{equation*}
\big[b,g^*_{\lambda,\alpha}\big](f)(x)=\left(\iint_{{\mathbb R}^{n+1}_+}\left(\frac t{t+|x-y|}\right)^{\lambda n}\sup_{\varphi\in{\mathcal C}_\alpha}\bigg|\int_{\mathbb R^n}\big[b(x)-b(z)\big]\varphi_t(y-z)f(z)\,dz\bigg|^2\frac{dydt}{t^{n+1}}\right)^{1/2}.
\end{equation*}

The classical Morrey spaces $\mathcal L^{p,\lambda}$ were first introduced by Morrey in \cite{morrey} to study the
local behavior of solutions to second order elliptic partial differential equations. For the boundedness of the
Hardy-Littlewood maximal operator, the fractional integral operator and the Calder\'on-Zygmund singular integral
operator on these spaces, we refer the readers to \cite{adams,chiarenza,peetre}. For the properties and applications
of classical Morrey spaces, see \cite{fan,fazio1,fazio2} and the references therein.

In 2009, Komori and Shirai \cite{komori} first defined the weighted Morrey spaces $L^{p,\kappa}(w)$ which could be
viewed as an extension of weighted Lebesgue spaces, and studied the boundedness of the above classical operators on
these weighted spaces. Recently, in \cite{wang1}, \cite{wang5} and \cite{wang2}, we have established the continuity properties of some other operators on the weighted Morrey spaces $L^{p,\kappa}(w)$.

The purpose of this paper is to discuss the boundedness properties of intrinsic square functions and their commutators on the weighted Morrey spaces $L^{p,\kappa}(w)$ for all $1<p<\infty$ and $0<\kappa<1$. Our main results in the paper are formulated as follows.
\newtheorem{theorem}{Theorem}[section]

\begin{theorem}
Let $0<\alpha\le1$, $1<p<\infty$, $0<\kappa<1$ and $w\in A_p$. Then there is a
constant $C>0$ independent of $f$ such that
\begin{equation*}
\big\|\mathcal S_\alpha(f)\big\|_{L^{p,\kappa}(w)}\le C\|f\|_{L^{p,\kappa}(w)}.
\end{equation*}
\end{theorem}

\begin{theorem}
Let $0<\alpha\le1$, $1<p<\infty$, $0<\kappa<1$ and $w\in A_p$. Suppose that $b\in BMO(\mathbb R^n)$, then there is a
constant $C>0$ independent of $f$ such that
\begin{equation*}
\big\|\big[b,\mathcal S_\alpha\big](f)\big\|_{L^{p,\kappa}(w)}\le C\|f\|_{L^{p,\kappa}(w)}.
\end{equation*}
\end{theorem}

\begin{theorem}
Let $0<\alpha\le1$, $1<p<\infty$, $0<\kappa<1$ and $w\in A_p$. If $\lambda>\max\{p,3\}$, then there is a
constant $C>0$ independent of $f$ such that
\begin{equation*}
\big\|g^*_{\lambda,\alpha}(f)\big\|_{L^{p,\kappa}(w)}\le C\|f\|_{L^{p,\kappa}(w)}.
\end{equation*}
\end{theorem}

\begin{theorem}
Let $0<\alpha\le1$, $1<p<\infty$, $0<\kappa<1$ and $w\in A_p$. If $b\in BMO(\mathbb R^n)$ and $\lambda>\max\{p,3\}$, then there is a
constant $C>0$ independent of $f$ such that
\begin{equation*}
\big\|\big[b,g^*_{\lambda,\alpha}\big](f)\big\|_{L^{p,\kappa}(w)}\le C\|f\|_{L^{p,\kappa}(w)}.
\end{equation*}
\end{theorem}

In \cite{wilson1}, Wilson also showed that for any $0<\alpha\le1$, the functions $\mathcal S_\alpha(f)(x)$ and $g_\alpha(f)(x)$ are pointwise comparable. Thus, as a direct consequence of Theorems 1.1 and 1.2, we obtain the following

\newtheorem{corollary}[theorem]{Corollary}

\begin{corollary}
Let $0<\alpha\le1$, $1<p<\infty$, $0<\kappa<1$ and $w\in A_p$. Then there is a
constant $C>0$ independent of $f$ such that
\begin{equation*}
\big\|g_\alpha(f)\big\|_{L^{p,\kappa}(w)}\le C\|f\|_{L^{p,\kappa}(w)}.
\end{equation*}
\end{corollary}

\begin{corollary}
Let $0<\alpha\le1$, $1<p<\infty$, $0<\kappa<1$ and $w\in A_p$. Suppose that $b\in BMO(\mathbb R^n)$, then there is a
constant $C>0$ independent of $f$ such that
\begin{equation*}
\big\|\big[b,g_\alpha\big](f)\big\|_{L^{p,\kappa}(w)}\le C\|f\|_{L^{p,\kappa}(w)}.
\end{equation*}
\end{corollary}

\section{Notations and definitions}

The classical $A_p$ weight theory was first introduced by Muckenhoupt in the study of weighted $L^p$ boundedness of Hardy-Littlewood maximal functions in \cite{muckenhoupt1}. A weight $w$ is a nonnegative, locally integrable function on $\mathbb R^n$, $B=B(x_0,r_B)$ denotes the ball with the center $x_0$ and radius $r_B$. Given a ball $B$ and $\lambda>0$, $\lambda B$ denotes the ball with the same center as $B$ whose radius is $\lambda$ times that of $B$. For a given weight function $w$ and a measurable set $E$, we also denote the Lebesgue measure of $E$ by $|E|$ and the weighted measure of $E$ by $w(E)$, where $w(E)=\int_E w(x)\,dx$. We say that $w$ is in the Muckenhoupt class $A_p$ with $1<p<\infty$, if there exists a constant $C>0$ such that for every ball $B\subseteq \mathbb R^n$,
\begin{equation*}
\left(\frac1{|B|}\int_B w(x)\,dx\right)\left(\frac1{|B|}\int_B w(x)^{-1/{(p-1)}}\,dx\right)^{p-1}\le C.
\end{equation*}
The smallest constant $C$ such that the above inequality holds is called the $A_p$ characteristic constant of $w$ and denoted by $[w]_{A_p}$. A weight function $w$ is said to belong to the reverse H\"{o}lder class $RH_r$ if there exist two constants $r>1$ and
$C>0$ such that the following reverse H\"{o}lder inequality holds for every
ball $B\subseteq \mathbb R^n$.
\begin{equation*}
\left(\frac{1}{|B|}\int_B w(x)^r\,dx\right)^{1/r}\le C\left(\frac{1}{|B|}\int_B w(x)\,dx\right).
\end{equation*}

We state the following results that we will use frequently in the sequel.

\newtheorem{lemma}[theorem]{Lemma}

\begin{lemma}[\cite{garcia2}]
Let $w\in A_p$ with $1<p<\infty$. Then, for any ball $B$, there exists an absolute constant $C>0$ such that
\begin{equation*}
w(2B)\le C\,w(B).
\end{equation*}
In general, for any $\lambda>1$, we have
\begin{equation*}
w(\lambda B)\le C\cdot\lambda^{np}w(B),
\end{equation*}
where $C$ does not depend on $B$ nor on $\lambda$.
\end{lemma}

\begin{lemma}[\cite{gundy}]
Let $w\in RH_r$ with $r>1$. Then there exists a constant $C>0$ such that
\begin{equation*}
\frac{w(E)}{w(B)}\le C\left(\frac{|E|}{|B|}\right)^{(r-1)/r}
\end{equation*}
for any measurable subset $E$ of a ball $B$.
\end{lemma}

Given a weight function $w$ on $\mathbb R^n$, for $1<p<\infty$, we denote by $L^p_w(\mathbb R^n)$ the space of all functions satisfying
\begin{equation*}
\|f\|_{L^p_w}=\bigg(\int_{\mathbb R^n}|f(x)|^pw(x)\,dx\bigg)^{1/p}<\infty.
\end{equation*}

A locally integrable function $b$ is said to be in $BMO(\mathbb R^n)$ if
\begin{equation*}
\|b\|_*=\sup_{B}\frac{1}{|B|}\int_B|b(x)-b_B|\,dx<\infty,
\end{equation*}
where $b_B$ stands for the average of $b$ on $B$, i.e., $b_B=\frac{1}{|B|}\int_B b(y)\,dy$ and the supremum is taken
over all balls $B$ in $\mathbb R^n$.

\begin{theorem}[\cite{duoand,john}]
Assume that $b\in BMO(\mathbb R^n)$. Then for any $1\le p<\infty$, we have
\begin{equation*}
\sup_B\bigg(\frac{1}{|B|}\int_B\big|b(x)-b_B\big|^p\,dx\bigg)^{1/p}\le C\|b\|_*.
\end{equation*}
\end{theorem}

\newtheorem{defn}[theorem]{Definition}
\begin{defn}[\cite{komori}]
Let $1\le p<\infty$, $0<\kappa<1$ and $w$ be a weight function. Then the weighted Morrey space is defined by
\begin{equation*}
L^{p,\kappa}(w)=\big\{f\in L^p_{loc}(w):\|f\|_{L^{p,\kappa}(w)}<\infty\big\},
\end{equation*}
where
\begin{equation*}
\|f\|_{L^{p,\kappa}(w)}=\sup_B\left(\frac{1}{w(B)^\kappa}\int_B|f(x)|^pw(x)\,dx\right)^{1/p}
\end{equation*}
and the supremum is taken over all balls $B$ in $\mathbb R^n$.
\end{defn}
Throughout this article, we will use $C$ to denote a positive constant, which is independent of the main parameters
and not necessarily the same at each occurrence. Moreover, we will denote the conjugate exponent of $p>1$ by $p'=p/{(p-1)}$.

\section{Proofs of Theorems 1.1 and 1.2}

\begin{proof}[Proof of Theorem 1.1]
Fix a ball $B=B(x_0,r_B)\subseteq\mathbb R^n$ and decompose $f=f_1+f_2$, where $f_1=f\chi_{_{2B}}$, $\chi_{_{2B}}$ denotes the characteristic function of $2B$. Since $\mathcal S_\alpha$($0<\alpha\le1$) is a sublinear operator, then we have
\begin{equation*}
\begin{split}
&\frac{1}{w(B)^{\kappa/p}}\bigg(\int_B|\mathcal S_\alpha(f)(x)|^pw(x)\,dx\bigg)^{1/p}\\
\le\,&\frac{1}{w(B)^{\kappa/p}}\bigg(\int_B|\mathcal S_\alpha(f_1)(x)|^pw(x)\,dx\bigg)^{1/p}+
\frac{1}{w(B)^{\kappa/p}}\bigg(\int_B|\mathcal S_\alpha(f_2)(x)|^pw(x)\,dx\bigg)^{1/p}\\
=\,&I_1+I_2.
\end{split}
\end{equation*}
Theorem A and Lemma 2.1 imply
\begin{equation*}
\begin{split}
I_1&\le C\cdot\frac{1}{w(B)^{\kappa/p}}\bigg(\int_{2B}|f(x)|^pw(x)\,dx\bigg)^{1/p}\\
&\le C\|f\|_{L^{p,\kappa}(w)}\cdot\frac{w(2B)^{\kappa/p}}{w(B)^{\kappa/p}}\\
&\le C\|f\|_{L^{p,\kappa}(w)}.
\end{split}
\end{equation*}
We now turn to estimate the other term $I_2$. For any $\varphi\in{\mathcal C}_\alpha$, $0<\alpha\le1$ and $(y,t)\in\Gamma(x)$, we have
\begin{align}
\big|f_2*\varphi_t(y)\big|&=\bigg|\int_{(2B)^c}\varphi_t(y-z)f(z)\,dz\bigg|\notag\\
&\le C\cdot t^{-n}\int_{(2B)^c\cap\{z:|y-z|\le t\}}|f(z)|\,dz\notag\\
&\le C\cdot t^{-n}\sum_{k=1}^\infty\int_{(2^{k+1}B\backslash 2^{k}B)\cap\{z:|y-z|\le t\}}|f(z)|\,dz.
\end{align}
For any $x\in B$, $(y,t)\in\Gamma(x)$ and $z\in\big(2^{k+1}B\backslash 2^{k}B\big)\cap B(y,t)$, then by a direct computation, we can easily see that
\begin{equation*}
2t\ge |x-y|+|y-z|\ge|x-z|\ge|z-x_0|-|x-x_0|\ge 2^{k-1}r_B.
\end{equation*}
Thus, by using the above inequality (1) and Minkowski's integral inequality, we deduce
\begin{equation*}
\begin{split}
\big|\mathcal S_\alpha(f_2)(x)\big|&=\left(\iint_{\Gamma(x)}\sup_{\varphi\in{\mathcal C}_\alpha}|f_2*\varphi_t(y)|^2\frac{dydt}{t^{n+1}}\right)^{1/2}\\
&\le C\left(\int_{2^{k-2}r_B}^\infty\int_{|x-y|<t}\bigg|t^{-n}\sum_{k=1}^\infty\int_{2^{k+1}B\backslash 2^{k}B}|f(z)|\,dz\bigg|^2\frac{dydt}{t^{n+1}}\right)^{1/2}\\
&\le C\bigg(\sum_{k=1}^\infty\int_{2^{k+1}B\backslash 2^{k}B}|f(z)|\,dz\bigg)\bigg(\int_{2^{k-2}r_B}^\infty\frac{dt}{t^{2n+1}}\bigg)^{1/2}\\
&\le C\sum_{k=1}^\infty\frac{1}{|2^{k+1}B|}\int_{2^{k+1}B\backslash 2^{k}B}|f(z)|\,dz.
\end{split}
\end{equation*}
It follows from H\"older's inequality and the $A_p$ condition that
\begin{align}
\frac{1}{|2^{k+1}B|}\int_{2^{k+1}B}|f(z)|\,dz&\le\frac{1}{|2^{k+1}B|}\bigg(\int_{2^{k+1}B}|f(z)|^pw(z)\,dz\bigg)^{1/p}
\bigg(\int_{2^{k+1}B}w(z)^{-{p'}/p}\,dz\bigg)^{1/{p'}}\notag\\
&\le C\|f\|_{L^{p,\kappa}(w)}\cdot w\big(2^{k+1}B\big)^{(\kappa-1)/p}.
\end{align}
Hence
\begin{equation*}
I_2\le C\|f\|_{L^{p,\kappa}(w)}\sum_{k=1}^\infty\frac{w(B)^{(1-\kappa)/p}}{w(2^{k+1}B)^{(1-\kappa)/p}}.
\end{equation*}
Since $w\in A_p$ with $1<p<\infty$, then there exists a number $r>1$ such that $w\in RH_r$. Consequently, by using Lemma 2.2, we can get
\begin{equation}
\frac{w(B)}{w(2^{k+1}B)}\le C\left(\frac{|B|}{|2^{k+1}B|}\right)^{(r-1)/r}.
\end{equation}
Therefore
\begin{equation*}
\begin{split}
I_2&\le C\|f\|_{L^{p,\kappa}(w)}\sum_{k=1}^\infty\left(\frac{1}{2^{kn}}\right)^{{(1-\kappa)(r-1)}/{pr}}\\
&\le C\|f\|_{L^{p,\kappa}(w)},
\end{split}
\end{equation*}
where the last series is convergent since ${(1-\kappa)(r-1)}/{pr}>0$. Combining the above estimates for $I_1$ and $I_2$ and taking the supremum over all balls $B\subseteq\mathbb R^n$, we complete the proof of Theorem 1.1.
\end{proof}

Given a real-valued function $b\in BMO(\mathbb R^n)$, we shall follow the idea developed in \cite{alvarez,ding} and denote $F(\xi)=e^{\xi[b(x)-b(z)]}$, $\xi\in\mathbb C$. Then by the analyticity of $F(\xi)$ on $\mathbb C$ and the Cauchy integral formula, we get
\begin{equation*}
\begin{split}
b(x)-b(z)=F'(0)&=\frac{1}{2\pi i}\int_{|\xi|=1}\frac{F(\xi)}{\xi^2}\,d\xi\\
&=\frac{1}{2\pi}\int_0^{2\pi}e^{e^{i\theta}[b(x)-b(z)]}e^{-i\theta}\,d\theta.
\end{split}
\end{equation*}
Thus, for any $\varphi\in{\mathcal C}_\alpha$, $0<\alpha\le1$, we obtain
\begin{equation*}
\begin{split}
\bigg|\int_{\mathbb R^n}\big[b(x)-b(z)\big]\varphi_t(y-z)f(z)\,dz\bigg|&=
\bigg|\frac{1}{2\pi}\int_0^{2\pi}\bigg(\int_{\mathbb R^n}\varphi_t(y-z)e^{-e^{i\theta}b(z)}f(z)\,dz\bigg)
e^{e^{i\theta}b(x)}e^{-i\theta}\,d\theta\bigg|\\
&\le\frac{1}{2\pi}\int_0^{2\pi}\sup_{\varphi\in{\mathcal C}_\alpha}\bigg|\int_{\mathbb R^n}\varphi_t(y-z)e^{-e^{i\theta}b(z)}f(z)\,dz\bigg|e^{\cos\theta\cdot b(x)}\,d\theta\\
&\le\frac{1}{2\pi}\int_0^{2\pi}A_\alpha\big(e^{-e^{i\theta}b}\cdot f\big)(y,t)\cdot e^{\cos\theta\cdot b(x)}\,d\theta.
\end{split}
\end{equation*}
So we have
\begin{equation*}
\big|\big[b,\mathcal S_\alpha\big](f)(x)\big|\le\frac{1}{2\pi}\int_0^{2\pi}
\mathcal S_\alpha\big(e^{-e^{i\theta}b}\cdot f\big)(x)\cdot e^{\cos\theta\cdot b(x)}\,d\theta,
\end{equation*}
\begin{equation*}
\big|\big[b,g^*_{\lambda,\alpha}\big](f)(x)\big|\le\frac{1}{2\pi}\int_0^{2\pi}
g^*_{\lambda,\alpha}\big(e^{-e^{i\theta}b}\cdot f\big)(x)\cdot e^{\cos\theta\cdot b(x)}\,d\theta.
\end{equation*}
Then, by using the same arguments as in \cite{ding}, we can also show the following
\begin{theorem}
Let $0<\alpha\le1$, $1<p<\infty$ and $w\in A_p$. Then the commutators $\big[b,\mathcal S_\alpha\big]$ and $\big[b,g^*_{\lambda,\alpha}\big]$ are all bounded from $L^p_w(\mathbb R^n)$ into itself whenever $b\in BMO(\mathbb R^n)$.
\end{theorem}

\begin{proof}[Proof of Theorem 1.2]
Fix a ball $B=B(x_0,r_B)\subseteq\mathbb R^n$. Let $f=f_1+f_2$, where $f_1=f\chi_{_{2B}}$. Then we can write
\begin{equation*}
\begin{split}
&\frac{1}{w(B)^{\kappa/p}}\bigg(\int_B\big|\big[b,\mathcal S_\alpha\big](f)(x)\big|^pw(x)\,dx\bigg)^{1/p}\\
\le\,&\frac{1}{w(B)^{\kappa/p}}\bigg(\int_B\big|\big[b,\mathcal S_\alpha\big](f_1)(x)\big|^pw(x)\,dx\bigg)^{1/p}+
\frac{1}{w(B)^{\kappa/p}}\bigg(\int_B\big|\big[b,\mathcal S_\alpha\big](f_2)(x)\big|^pw(x)\,dx\bigg)^{1/p}\\
=\,&J_1+J_2.
\end{split}
\end{equation*}
Applying Theorem 3.1 and Lemma 2.1, we thus obtain
\begin{align}
J_1&\le C\cdot\frac{1}{w(B)^{\kappa/p}}\bigg(\int_{2B}|f(x)|^pw(x)\,dx\bigg)^{1/p}\notag\\
&\le C\|f\|_{L^{p,\kappa}(w)}\cdot\frac{w(2B)^{\kappa/p}}{w(B)^{\kappa/p}}\notag\\
&\le C\|f\|_{L^{p,\kappa}(w)}.
\end{align}
We now turn to deal with the term $J_2$. For any given $x\in B$ and $(y,t)\in\Gamma(x)$, we have
\begin{equation*}
\begin{split}
\sup_{\varphi\in{\mathcal C}_\alpha}\bigg|\int_{\mathbb R^n}\big[b(x)-b(z)\big]\varphi_t(y-z)f_2(z)\,dz\bigg|&\le
\big|b(x)-b_B\big|\cdot\sup_{\varphi\in{\mathcal C}_\alpha}\bigg|\int_{\mathbb R^n}\varphi_t(y-z)f_2(z)\,dz\bigg|\\
&+\sup_{\varphi\in{\mathcal C}_\alpha}\bigg|\int_{\mathbb R^n}\big[b(z)-b_B\big]\varphi_t(y-z)f_2(z)\,dz\bigg|
\end{split}
\end{equation*}
Hence
\begin{equation*}
\begin{split}
\big|\big[b,\mathcal S_\alpha\big](f_2)(x)\big|&\le\big|b(x)-b_B\big|\cdot\mathcal S_\alpha(f_2)(x)\\
&+\left(\iint_{\Gamma(x)}\sup_{\varphi\in{\mathcal C}_\alpha}\bigg|\int_{\mathbb R^n}\big[b(z)-b_B\big]\varphi_t(y-z)f_2(z)\,dz\bigg|^2\frac{dydt}{t^{n+1}}\right)^{1/2}\\
&=\mbox{\upshape I+II}.
\end{split}
\end{equation*}
In the proof of Theorem 1.1, we have already proved that for any $x\in B$,
\begin{equation*}
\big|\mathcal S_\alpha(f_2)(x)\big|\le C\|f\|_{L^{p,\kappa}(w)}\cdot \sum_{k=1}^\infty w\big(2^{k+1}B\big)^{(\kappa-1)/p}.
\end{equation*}
Consequently
\begin{equation*}
\begin{split}
&\frac{1}{w(B)^{\kappa/p}}\bigg(\int_B \mbox{\upshape I}^p\,w(x)\,dx\bigg)^{1/p}\\
\le \,&C\|f\|_{L^{p,\kappa}(w)}\frac{1}{w(B)^{\kappa/p}}\cdot\sum_{k=1}^\infty w\big(2^{k+1}B\big)^{(\kappa-1)/p}\cdot\bigg(\int_B\big|b(x)-b_B\big|^pw(x)\,dx\bigg)^{1/p}\\
=\,&C\|f\|_{L^{p,\kappa}(w)}\sum_{k=1}^\infty\frac{w(B)^{(1-\kappa)/p}}{w(2^{k+1}B)^{(1-\kappa)/p}}
\cdot\bigg(\frac{1}{w(B)}\int_B\big|b(x)-b_B\big|^pw(x)\,dx\bigg)^{1/p}.
\end{split}
\end{equation*}
Using the same arguments as that of Theorem 1.1, we can see that the above summation is bounded by a constant. Hence
\begin{equation*}
\frac{1}{w(B)^{\kappa/p}}\bigg(\int_B \mbox{\upshape I}^p\,w(x)\,dx\bigg)^{1/p}\le C\|f\|_{L^{p,\kappa}(w)}\bigg(\frac{1}{w(B)}\int_B\big|b(x)-b_B\big|^pw(x)\,dx\bigg)^{1/p}.
\end{equation*}
Since $w\in A_p$, as before, we know that there exists a number $r>1$ such that $w\in RH_r$. Thus by H\"older's inequality and Theorem 2.3, we deduce
\begin{align}
&\bigg(\frac{1}{w(B)}\int_B\big|b(x)-b_B\big|^pw(x)\,dx\bigg)^{1/p}\notag\\
\le\, &\frac{1}{w(B)^{1/p}}\bigg(\int_B\big|b(x)-b_B\big|^{pr'}\,dx\bigg)^{1/{(pr')}}
\bigg(\int_Bw(x)^r\,dx\bigg)^{1/{(pr)}}\notag\\
\le\, &C\cdot\bigg(\frac{1}{|B|}\int_B\big|b(x)-b_B\big|^{pr'}\,dx\bigg)^{1/{(pr')}}\notag\\
\le\, &C\|b\|_*.
\end{align}
So we have
\begin{equation}
\frac{1}{w(B)^{\kappa/p}}\bigg(\int_B \mbox{\upshape I}^p\,w(x)\,dx\bigg)^{1/p}\le C\|b\|_*\|f\|_{L^{p,\kappa}(w)}.
\end{equation}
On the other hand
\begin{equation*}
\begin{split}
\mbox{\upshape II}&=\left(\iint_{\Gamma(x)}\sup_{\varphi\in{\mathcal C}_\alpha}\bigg|\int_{(2B)^c}\big[b(z)-b_B\big]\varphi_t(y-z)f(z)\,dz\bigg|^2\frac{dydt}{t^{n+1}}\right)^{1/2}\\
&\le C\left(\iint_{\Gamma(x)}
\bigg|t^{-n}\sum_{k=1}^\infty\int_{(2^{k+1}B\backslash 2^{k}B)\cap\{z:|y-z|\le t\}}|b(z)-b_B||f(z)|\,dz\bigg|^2\frac{dydt}{t^{n+1}}\right)^{1/2}\\
&\le C\left(\iint_{\Gamma(x)}
\bigg|t^{-n}\sum_{k=1}^\infty\int_{(2^{k+1}B\backslash 2^{k}B)\cap\{z:|y-z|\le t\}}\big|b(z)-b_{2^{k+1}B}\big||f(z)|\,dz\bigg|^2\frac{dydt}{t^{n+1}}\right)^{1/2}\\
&+C\left(\iint_{\Gamma(x)}
\bigg|t^{-n}\sum_{k=1}^\infty\big|b_{2^{k+1}B}-b_B\big|\cdot\int_{(2^{k+1}B\backslash 2^{k}B)\cap\{z:|y-z|\le t\}}|f(z)|\,dz\bigg|^2\frac{dydt}{t^{n+1}}\right)^{1/2}\\
&= \mbox{\upshape III+IV}.
\end{split}
\end{equation*}
An application of H\"older's inequality gives us that
\begin{align}
&\int_{2^{k+1}B\backslash 2^{k}B}\big|b(z)-b_{2^{k+1}B}\big||f(z)|\,dz\notag\\
\le&\,\bigg(\int_{2^{k+1}B}\big|b(z)-b_{2^{k+1}B}\big|^{p'}w(z)^{-{p'}/p}\,dz\bigg)^{1/{p'}}
\bigg(\int_{2^{k+1}B}\big|f(z)\big|^pw(z)\,dz\bigg)^{1/p}\notag\\
\le&\,C\|f\|_{L^{p,\kappa}(w)}\cdot w\big(2^{k+1}B\big)^{\kappa/p}
\bigg(\int_{2^{k+1}B}\big|b(z)-b_{2^{k+1}B}\big|^{p'}w(z)^{-{p'}/p}\,dz\bigg)^{1/{p'}}.
\end{align}
If we set $\nu(z)=w(z)^{-{p'}/p}=w(z)^{1-p'}$, then we have $\nu\in A_{p'}$ because $w\in A_p$ (see \cite{garcia2}). Following along the same lines as in the proof of (5), we can also show
\begin{equation}
\bigg(\frac{1}{\nu(2^{k+1}B)}\int_{2^{k+1}B}\big|b(z)-b_{2^{k+1}B}\big|^{p'}\nu(z)\,dz\bigg)^{1/{p'}}\le C\|b\|_*.
\end{equation}
Substituting the above inequality (8) into (7), we thus obtain
\begin{equation*}
\begin{split}
\int_{2^{k+1}B}\big|b(z)-b_{2^{k+1}B}\big||f(z)|\,dz&\le C\|b\|_*\|f\|_{L^{p,\kappa}(w)}\cdot w\big(2^{k+1}B\big)^{\kappa/p}\nu\big(2^{k+1}B\big)^{1/{p'}}\\
&\le C\|b\|_*\|f\|_{L^{p,\kappa}(w)}\cdot\big|2^{k+1}B\big|w\big(2^{k+1}B\big)^{{(\kappa-1)}/p}.
\end{split}
\end{equation*}
In addition, we note that in this case, $t\ge2^{k-2}r_B$ as in Theorem 1.1. From the above inequality, it follows that
\begin{equation*}
\begin{split}
\mbox{\upshape III}&\le C\left(\int_{2^{k-2}r_B}^\infty\int_{|x-y|<t}\bigg|t^{-n}\sum_{k=1}^\infty\int_{2^{k+1}B\backslash 2^{k}B}\big|b(z)-b_{2^{k+1}B}\big||f(z)|\,dz\bigg|^2\frac{dydt}{t^{n+1}}\right)^{1/2}\\
&\le C\bigg(\sum_{k=1}^\infty\int_{2^{k+1}B\backslash 2^{k}B}\big|b(z)-b_{2^{k+1}B}\big||f(z)|\,dz\bigg)\bigg(\int_{2^{k-2}r_B}^\infty\frac{dt}{t^{2n+1}}\bigg)^{1/2}\\
&\le C\|b\|_*\|f\|_{L^{p,\kappa}(w)}\cdot w\big(2^{k+1}B\big)^{{(\kappa-1)}/p}.
\end{split}
\end{equation*}
Hence
\begin{align}
\frac{1}{w(B)^{\kappa/p}}\bigg(\int_B \mbox{\upshape III}^p\,w(x)\,dx\bigg)^{1/p}&\le C\|b\|_*\|f\|_{L^{p,\kappa}(w)}\sum_{k=1}^\infty\frac{w(B)^{(1-\kappa)/p}}{w(2^{k+1}B)^{(1-\kappa)/p}}\notag\\
&\le C\|b\|_*\|f\|_{L^{p,\kappa}(w)}.
\end{align}
Now let us deal with the last term \mbox{\upshape IV}. Since $b\in BMO(\mathbb R^n)$, then a simple calculation shows that
\begin{equation}
\big|b_{2^{k+1}B}-b_B\big|\le C\cdot(k+1)\|b\|_*.
\end{equation}
It follows from the inequalities (2) and (10) that
\begin{equation*}
\begin{split}
\mbox{\upshape IV}&\le C\left(\int_{2^{k-2}r_B}^\infty\int_{|x-y|<t}\bigg|t^{-n}\sum_{k=1}^\infty\big|b_{2^{k+1}B}-b_B\big|
\cdot\int_{2^{k+1}B\backslash 2^{k}B}|f(z)|\,dz\bigg|^2\frac{dydt}{t^{n+1}}\right)^{1/2}\\
&\le C\|b\|_*\bigg(\sum_{k=1}^\infty(k+1)\cdot\int_{2^{k+1}B\backslash 2^{k}B}|f(z)|\,dz\bigg)\bigg(\int_{2^{k-2}r_B}^\infty\frac{dt}{t^{2n+1}}\bigg)^{1/2}\\
&\le C\|b\|_*\|f\|_{L^{p,\kappa}(w)}\sum_{k=1}^\infty(k+1)\cdot w\big(2^{k+1}B\big)^{{(\kappa-1)}/p}.
\end{split}
\end{equation*}
Therefore
\begin{align}
\frac{1}{w(B)^{\kappa/p}}\bigg(\int_B \mbox{\upshape IV}^p\,w(x)\,dx\bigg)^{1/p}&\le C\|b\|_*\|f\|_{L^{p,\kappa}(w)}\sum_{k=1}^\infty (k+1)\cdot\frac{w(B)^{(1-\kappa)/p}}{w(2^{k+1}B)^{(1-\kappa)/p}}\notag\\
&\le C\|b\|_*\|f\|_{L^{p,\kappa}(w)}\sum_{k=1}^\infty\frac{k}{2^{kn\theta}}\notag\\
&\le C\|b\|_*\|f\|_{L^{p,\kappa}(w)},
\end{align}
where we have used the previous estimate (3) with $w\in RH_r$ and $\theta=(1-\kappa)(r-1)/{pr}$. Summarizing the estimates (9) and (11) derived above, we thus obtain
\begin{equation}
\frac{1}{w(B)^{\kappa/p}}\bigg(\int_B \mbox{\upshape II}^p\,w(x)\,dx\bigg)^{1/p}\le C\|b\|_*\|f\|_{L^{p,\kappa}(w)}.
\end{equation}
Combining the inequalities (4), (6) with the above inequality (12) and taking the supremum over all balls $B\subseteq\mathbb R^n$, we complete the proof of Theorem 1.2.
\end{proof}

\section{Proofs of Theorems 1.3 and 1.4}

In order to prove the main theorems of this section, we need to establish the following three lemmas.

\begin{lemma}
Let $0<\alpha\le1$ and $w\in A_p$ with $p=2$. Then for any $j\in\mathbb Z_+$, we have
\begin{equation*}
\big\|\mathcal S_{\alpha,2^j}(f)\big\|_{L^2_w}\le C\cdot2^{jn}\big\|\mathcal S_\alpha(f)\big\|_{L^2_w}.
\end{equation*}
\end{lemma}

\begin{proof}
Since $w\in A_2$, then by Lemma 2.1, we get
\begin{equation*}
w\big(B(y,2^jt)\big)=w\big(2^jB(y,t)\big)\le C\cdot2^{2jn}w\big(B(y,t)\big)\quad j=1,2,\ldots.
\end{equation*}
Therefore
\begin{equation*}
\begin{split}
\big\|\mathcal S_{\alpha,2^j}(f)\big\|_{L^2_w}^2&=\int_{\mathbb R^n}\bigg(\iint_{{\mathbb R}^{n+1}_+}\Big(A_\alpha(f)(y,t)\Big)^2\chi_{|x-y|<2^j t}\frac{dydt}{t^{n+1}}\bigg)w(x)\,dx\\
&=\iint_{{\mathbb R}^{n+1}_+}\Big(\int_{|x-y|<2^j t}w(x)\,dx\Big)\Big(A_\alpha(f)(y,t)\Big)^2\frac{dydt}{t^{n+1}}\\
&\le C\cdot2^{2jn}\iint_{{\mathbb R}^{n+1}_+}\Big(\int_{|x-y|<t}w(x)\,dx\Big)\Big(A_\alpha(f)(y,t)\Big)^2\frac{dydt}{t^{n+1}}\\
&=C\cdot 2^{2jn}\big\|\mathcal S_\alpha(f)\big\|_{L^2_w}^2.
\end{split}
\end{equation*}
This finishes the proof of Lemma 4.1.
\end{proof}

\begin{lemma}
Let $0<\alpha\le1$, $2<p<\infty$ and $w\in A_p$. Then for any $j\in\mathbb Z_+$, we have
\begin{equation*}
\big\|\mathcal S_{\alpha,2^j}(f)\big\|_{L^p_w}\le C\cdot2^{jnp/2}\big\|\mathcal S_\alpha(f)\big\|_{L^p_w}.
\end{equation*}
\end{lemma}

\begin{proof}
For any $j\in\mathbb Z_+$, it is easy to see that
\begin{equation*}
\big\|\mathcal S_{\alpha,2^j}(f)\big\|^2_{L^p_w}=\big\|\mathcal S_{\alpha,2^j}(f)^2\big\|_{L^{p/2}_w}.
\end{equation*}
Since $p/2>1$, then we have
\begin{align}
&\big\|\mathcal S_{\alpha,2^j}(f)^2\big\|_{L^{p/2}_w}\notag\\
=&\underset{\|g\|_{L_w^{(p/2)'}}\le1}{\sup}\left|\int_{\mathbb R^n}\mathcal S_{\alpha,2^j}(f)(x)^2g(x)w(x)\,dx\right|\notag\\
=&\underset{\|g\|_{L_w^{(p/2)'}}\le1}{\sup}\left|\int_{\mathbb R^n}\bigg(\iint_{{\mathbb R}^{n+1}_+}\Big(A_\alpha(f)(y,t)\Big)^2\chi_{|x-y|<2^j t}\frac{dydt}{t^{n+1}}\bigg)g(x)w(x)\,dx\right|\notag\\
=&\underset{\|g\|_{L_w^{(p/2)'}}\le1}{\sup}\left|\iint_{{\mathbb R}^{n+1}_+}\Big(\int_{|x-y|<2^jt}g(x)w(x)\,dx\Big)\Big(A_\alpha(f)(y,t)\Big)^2 \frac{dydt}{t^{n+1}}\right|.
\end{align}
For $w\in A_p$, we denote the weighted maximal operator by $M_w$; that is
\begin{equation*}
M_w(f)(x)=\underset{x\in B}{\sup}\frac{1}{w(B)}\int_B|f(y)|w(y)\,dy,
\end{equation*}
where the supremum is taken over all balls $B$ which contain $x$. Then, by Lemma 2.1, we can get
\begin{align}
\int_{|x-y|<2^jt}g(x)w(x)\,dx&\le C\cdot2^{jnp}w\big(B(y,t)\big)\cdot\frac{1}{w(B(y,2^jt))}\int_{B(y,2^jt)}g(x)w(x)\,dx\notag\\
&\le C\cdot2^{jnp}w\big(B(y,t)\big)\underset{x\in B(y,t)}{\inf}M_w(g)(x)\notag\\
&\le C\cdot2^{jnp}\int_{|x-y|<t}M_w(g)(x)w(x)\,dx.
\end{align}
Substituting the above inequality (14) into (13) and using H\"older's inequality and the $L^{(p/2)'}_w$ boundedness of $M_w$, we thus obtain
\begin{equation*}
\begin{split}
\big\|\mathcal S_{\alpha,2^j}(f)^2\big\|_{L^{p/2}_w}&\le C\cdot2^{jnp}\underset{\|g\|_{L_w^{(p/2)'}}\le1}{\sup}
\left|\int_{\mathbb R^n}\mathcal S_\alpha(f)(x)^2M_w(g)(x)w(x)\,dx\right|\\
&\le C\cdot2^{jnp}\big\|\mathcal S_\alpha(f)^2\big\|_{L^{p/2}_w}\underset{\|g\|_{L_w^{(p/2)'}}\le1}
{\sup}\big\|M_w(g)\big\|_{L^{(p/2)'}_w}\\
&\le C\cdot2^{jnp}\big\|\mathcal S_\alpha(f)^2\big\|_{L^{p/2}_w}\\
&= C\cdot2^{jnp}\big\|\mathcal S_\alpha(f)\big\|^2_{L^p_w}.
\end{split}
\end{equation*}
This implies the desired result.
\end{proof}

\begin{lemma}
Let $0<\alpha\le1$, $1<p<2$ and $w\in A_p$. Then for any $j\in\mathbb Z_+$, we have
\begin{equation*}
\big\|\mathcal S_{\alpha,2^j}(f)\big\|_{L^p_w}\le C\cdot2^{jn}\big\|\mathcal S_\alpha(f)\big\|_{L^p_w}.
\end{equation*}
\end{lemma}

\begin{proof}
We will adopt the same method as in [18, page 315--316]. For any $j\in\mathbb Z_+$, set $\Omega_\lambda=\big\{x\in\mathbb R^n:\mathcal S_\alpha(f)(x)>\lambda\big\}$ and $\Omega_{\lambda,j}=\big\{x\in\mathbb R^n:\mathcal S_{\alpha,2^j}(f)(x)>\lambda\big\}.$ We also set
\begin{equation*}
\Omega^*_\lambda=\Big\{x\in\mathbb R^n:M_w(\chi_{\Omega_\lambda})(x)>\frac{1}{2^{(jnp+1)}\cdot[w]_{A_p}}\Big\}.
\end{equation*}
Observe that $w\big(\Omega_{\lambda,j}\big)\le w\big(\Omega^*_\lambda\big)+w\big(\Omega_{\lambda,j}\cap(\mathbb R^n\backslash\Omega^*_\lambda)\big)$. Thus
\begin{equation*}
\begin{split}
\big\|\mathcal S_{\alpha,2^j}(f)\big\|^p_{L^p_w}&=\int_0^\infty p\lambda^{p-1}w\big(\Omega_{\lambda,j}\big)\,d\lambda\\
&\le\int_0^\infty p\lambda^{p-1}w\big(\Omega^*_\lambda\big)\,d\lambda+\int_0^\infty p\lambda^{p-1}w\big(\Omega_{\lambda,j}\cap(\mathbb R^n\backslash\Omega^*_\lambda)\big)\,d\lambda\\
&=\mbox{\upshape I+II}.
\end{split}
\end{equation*}
The weighted weak type estimate of $M_w$ yields
\begin{equation}
\mbox{\upshape I}\le C\cdot2^{jnp}\int_0^\infty p\lambda^{p-1}w(\Omega_\lambda)\,d\lambda=C\cdot2^{jnp}\big\|\mathcal S_\alpha(f)\big\|^p_{L^p_w}.
\end{equation}
To estimate II, we now claim that the following inequality holds.
\begin{equation}
\int_{\mathbb R^n\backslash\Omega^*_\lambda}\mathcal S_{\alpha,2^j}(f)(x)^2w(x)\,dx\le C\cdot2^{jnp}\int_{\mathbb R^n\backslash\Omega_\lambda}\mathcal S_{\alpha}(f)(x)^2w(x)\,dx.
\end{equation}
We will take the above inequality temporarily for granted, then it follows from Chebyshev's inequality and (16) that
\begin{equation*}
\begin{split}
w\big(\Omega_{\lambda,j}\cap(\mathbb R^n\backslash\Omega^*_\lambda)\big)&\le\lambda^{-2}\int_{\Omega_{\lambda,j}\cap(\mathbb R^n\backslash\Omega^*_\lambda)}\mathcal S_{\alpha,2^j}(f)(x)^2w(x)\,dx\\
&\le\lambda^{-2}\int_{\mathbb R^n\backslash\Omega^*_\lambda}\mathcal S_{\alpha,2^j}(f)(x)^2w(x)\,dx\\
&\le C\cdot2^{jnp}\lambda^{-2}\int_{\mathbb R^n\backslash\Omega_\lambda}\mathcal S_{\alpha}(f)(x)^2w(x)\,dx.
\end{split}
\end{equation*}
Hence
\begin{equation*}
\mbox{\upshape II}\le C\cdot2^{jnp}\int_0^\infty p\lambda^{p-1}\bigg(\lambda^{-2}\int_{\mathbb R^n\backslash\Omega_\lambda}\mathcal S_{\alpha}(f)(x)^2w(x)\,dx\bigg)d\lambda.
\end{equation*}
Changing the order of integration yields
\begin{align}
\mbox{\upshape II}&\le C\cdot2^{jnp}\int_{\mathbb R^n}\mathcal S_\alpha(f)(x)^2\bigg(\int_{|\mathcal S_\alpha(f)(x)|}^\infty p\lambda^{p-3}\,d\lambda\bigg)w(x)\,dx\notag\\
&\le C\cdot2^{jnp}\frac{p}{2-p}\cdot\big\|\mathcal S_\alpha(f)\big\|^p_{L^p_w}.
\end{align}
Combining the above estimate (17) with (15) and taking $p$-th root on both sides, we complete the proof of Lemma 4.3. So it remains to prove the inequality (16). Set $\Gamma_{2^j}(\mathbb R^n\backslash\Omega^*_\lambda)=\underset{x\in\mathbb R^n\backslash\Omega^*_\lambda}{\bigcup}\Gamma_{2^j}(x)$ and
$\Gamma(\mathbb R^n\backslash\Omega_\lambda)=\underset{x\in\mathbb R^n\backslash\Omega_\lambda}{\bigcup}\Gamma(x).$
For each given $(y,t)\in\Gamma_{2^j}(\mathbb R^n\backslash\Omega^*_\lambda)$, by Lemma 2.1, we thus have
\begin{equation*}
w\big(B(y,2^jt)\cap(\mathbb R^n\backslash\Omega^*_\lambda)\big)\le C\cdot2^{jnp}w\big(B(y,t)\big).
\end{equation*}
It is not difficult to check that $w\big(B(y,t)\cap\Omega_\lambda\big)\le\frac{w(B(y,t))}{2}$ and $\Gamma_{2^j}(\mathbb R^n\backslash\Omega^*_\lambda)\subseteq\Gamma(\mathbb R^n\backslash\Omega_\lambda)$. In fact, for any $(y,t)\in\Gamma_{2^j}(\mathbb R^n\backslash\Omega^*_\lambda)$, there exists a point $x\in \mathbb R^n\backslash\Omega^*_\lambda$ such that $(y,t)\in\Gamma_{2^j}(x)$. Then we can deduce
\begin{equation*}
\begin{split}
w\big(B(y,t)\cap\Omega_\lambda\big)&\le w\big(B(y,2^jt)\cap\Omega_\lambda\big)\\
&= \int_{B(y,2^jt)}\chi_{\Omega_\lambda}(z)w(z)\,dz\\
&\le [w]_{A_p}\cdot2^{jnp}w\big(B(y,t)\big)\cdot\frac{1}{w(B(y,2^jt))}\int_{B(y,2^jt)}\chi_{\Omega_\lambda}(z)w(z)\,dz.
\end{split}
\end{equation*}
Note that $x\in B(y,2^jt)\cap(\mathbb R^n\backslash\Omega^*_\lambda)$. So we have
\begin{equation*}
\begin{split}
w\big(B(y,t)\cap\Omega_\lambda\big)\le [w]_{A_p}\cdot2^{jnp}w\big(B(y,t)\big)M_w(\chi_{\Omega_\lambda})(x)\le \frac{w(B(y,t))}{2}.
\end{split}
\end{equation*}
Hence
\begin{equation*}
\begin{split}
w\big(B(y,t)\big)&=w\big(B(y,t)\cap\Omega_\lambda\big)+w\big(B(y,t)\cap(\mathbb R^n\backslash\Omega_\lambda)\big)\\
&\le \frac{w(B(y,t))}{2}+w\big(B(y,t)\cap(\mathbb R^n\backslash\Omega_\lambda)\big),
\end{split}
\end{equation*}
which is equivalent to
\begin{equation*}
w\big(B(y,t)\big)\le 2\cdot w\big(B(y,t)\cap(\mathbb R^n\backslash\Omega_\lambda)\big).
\end{equation*}
The above inequality implies in particular that there is a point $z\in B(y,t)\cap(\mathbb R^n\backslash\Omega_\lambda)\neq\emptyset$. In this case, we have $(y,t)\in\Gamma(z)$ with $z\in \mathbb R^n\backslash\Omega_\lambda$, which yields $\Gamma_{2^j}(\mathbb R^n\backslash\Omega^*_\lambda)\subseteq\Gamma(\mathbb R^n\backslash\Omega_\lambda)$. Thus we obtain
\begin{equation*}
w\big(B(y,2^jt)\cap(\mathbb R^n\backslash\Omega_\lambda^*)\big)\le C\cdot2^{jnp}w\big(B(y,t)\cap(\mathbb R^n\backslash\Omega_\lambda)\big).
\end{equation*}
Therefore
\begin{equation*}
\begin{split}
&\int_{\mathbb R^n\backslash\Omega^*_\lambda}\mathcal S_{\alpha,2^j}(f)(x)^2w(x)\,dx\\
=&\int_{\mathbb R^n\backslash\Omega^*_\lambda}\bigg(\iint_{\Gamma_{2^j}(x)}\Big(A_\alpha(f)(y,t)\Big)^2\frac{dydt}{t^{n+1}}
\bigg)w(x)\,dx\\
\le&\iint_{\Gamma_{2^j}(\mathbb R^n\backslash\Omega^*_\lambda)}\Big(\int_{B(y,2^jt)\cap(\mathbb R^n\backslash\Omega_\lambda^*)}w(x)\,dx\Big)\Big(A_\alpha(f)(y,t)\Big)^2\frac{dydt}{t^{n+1}}\\
\le&\,C\cdot2^{jnp}\iint_{\Gamma(\mathbb R^n\backslash\Omega_\lambda)}\Big(\int_{B(y,t)\cap(\mathbb R^n\backslash\Omega_{\lambda})}w(x)\,dx\Big)\Big(A_\alpha(f)(y,t)\Big)^2\frac{dydt}{t^{n+1}}\\
\le&\,C\cdot2^{jnp}\int_{\mathbb R^n\backslash\Omega_\lambda}\mathcal S_{\alpha}(f)(x)^2w(x)\,dx,
\end{split}
\end{equation*}
which is just our desired conclusion.
\end{proof}

We are now in a position to give the proof of Theorem 1.3.

\begin{proof}[Proof of Theorem 1.3]
From the definition of $g^*_{\lambda,\alpha}$, we readily see that
\begin{equation*}
\begin{split}
g^*_{\lambda,\alpha}(f)(x)^2=&\iint_{\mathbb R^{n+1}_+}\left(\frac{t}{t+|x-y|}\right)^{\lambda n}\Big(A_\alpha(f)(y,t)\Big)^2\frac{dydt}{t^{n+1}}\\
=&\int_0^\infty\int_{|x-y|<t}\left(\frac{t}{t+|x-y|}\right)^{\lambda n}\Big(A_\alpha(f)(y,t)\Big)^2\frac{dydt}{t^{n+1}}\\
&+\sum_{j=1}^\infty\int_0^\infty\int_{2^{j-1}t\le|x-y|<2^jt}\left(\frac{t}{t+|x-y|}\right)^{\lambda n}\Big(A_\alpha(f)(y,t)\Big)^2\frac{dydt}{t^{n+1}}\\
\le&\, C\bigg[\mathcal S_\alpha(f)(x)^2+\sum_{j=1}^\infty 2^{-j\lambda n}\mathcal S_{\alpha,2^j}(f)(x)^2\bigg].
\end{split}
\end{equation*}
For any given ball $B=B(x_0,r_B)\subseteq\mathbb R^n$, then from the above inequality, it follows that
\begin{equation*}
\begin{split}
&\frac{1}{w(B)^{\kappa/p}}\bigg(\int_B\big|g^*_{\lambda,\alpha}(f)(x)\big|^pw(x)\,dx\bigg)^{1/p}\\
\le\,&\frac{1}{w(B)^{\kappa/p}}\bigg(\int_B\big|\mathcal S_\alpha(f)(x)\big|^pw(x)\,dx\bigg)^{1/p}
+\sum_{j=1}^\infty 2^{-j\lambda n/2}\cdot
\frac{1}{w(B)^{\kappa/p}}\bigg(\int_B\big|\mathcal S_{\alpha,2^j}(f)(x)\big|^pw(x)\,dx\bigg)^{1/p}\\
=\,&I_0+\sum_{j=1}^\infty 2^{-j\lambda n/2}I_j.
\end{split}
\end{equation*}
By Theorem 1.1, we know that $I_0\le C\|f\|_{L^{p,\kappa}(w)}$. Below we shall give the estimates of $I_j$ for $j=1,2,\ldots.$ As before, we set $f=f_1+f_2$, $f_1=f\chi_{_{2B}}$ and write
\begin{equation*}
\begin{split}
&\frac{1}{w(B)^{\kappa/p}}\bigg(\int_B\big|\mathcal S_{\alpha,2^j}(f)(x)\big|^pw(x)\,dx\bigg)^{1/p}\\
\le\,&\frac{1}{w(B)^{\kappa/p}}\bigg(\int_B\big|\mathcal S_{\alpha,2^j}(f_1)(x)\big|^pw(x)\,dx\bigg)^{1/p}+
\frac{1}{w(B)^{\kappa/p}}\bigg(\int_B\big|\mathcal S_{\alpha,2^j}(f_2)(x)\big|^pw(x)\,dx\bigg)^{1/p}\\
=\,&I^{(1)}_j+I^{(2)}_j.
\end{split}
\end{equation*}
Applying Lemmas 4.1--4.3, Theorem A and Lemma 2.1, we obtain
\begin{equation*}
\begin{split}
I^{(1)}_j&\le \frac{1}{w(B)^{\kappa/p}}\big\|\mathcal S_{\alpha,2^j}(f_1)\big\|_{L^p_w}\\
&\le C\Big(2^{jn}+2^{jnp/2}\Big)\frac{1}{w(B)^{\kappa/p}}\cdot\|f_1\|_{L^p_w}\\
&\le C\|f\|_{L^{p,\kappa}(w)}\Big(2^{jn}+2^{jnp/2}\Big)\cdot\frac{w(2B)^{\kappa/p}}{w(B)^{\kappa/p}}\\
&\le C\|f\|_{L^{p,\kappa}(w)}\Big(2^{jn}+2^{jnp/2}\Big).
\end{split}
\end{equation*}
We now turn to estimate the term $I^{(2)}_j$. For any $x\in B$, $(y,t)\in\Gamma_{2^j}(x)$ and $z\in\big(2^{k+1}B\backslash 2^{k}B\big)\cap B(y,t)$, then by a direct calculation, we can easily deduce
\begin{equation*}
t+2^j t\ge |x-y|+|y-z|\ge|x-z|\ge|z-x_0|-|x-x_0|\ge 2^{k-1}r_B.
\end{equation*}
Thus, it follows from the previous estimates (1) and (2) that
\begin{equation*}
\begin{split}
\big|\mathcal S_{\alpha,2^j}(f_2)(x)\big|&=\left(\iint_{\Gamma_{2^j}(x)}\sup_{\varphi\in{\mathcal C}_\alpha}|f_2*\varphi_t(y)|^2\frac{dydt}{t^{n+1}}\right)^{1/2}\\
&\le C\left(\int_{2^{(k-2-j)}r_B}^\infty\int_{|x-y|<2^jt}\bigg|t^{-n}\sum_{k=1}^\infty\int_{2^{k+1}B\backslash 2^{k}B}|f(z)|\,dz\bigg|^2\frac{dydt}{t^{n+1}}\right)^{1/2}\\
&\le C\bigg(\sum_{k=1}^\infty\int_{2^{k+1}B\backslash 2^{k}B}|f(z)|\,dz\bigg)\bigg(\int_{2^{(k-2-j)}r_B}^\infty 2^{jn}\frac{dt}{t^{2n+1}}\bigg)^{1/2}\\
&\le C\cdot2^{{3jn}/2}\sum_{k=1}^\infty\frac{1}{|2^{k+1}B|}\int_{2^{k+1}B\backslash 2^{k}B}|f(z)|\,dz\\
&\le C\|f\|_{L^{p,\kappa}(w)}\cdot2^{{3jn}/2}\sum_{k=1}^\infty w\big(2^{k+1}B\big)^{(\kappa-1)/p}.
\end{split}
\end{equation*}
Furthermore, by using (3) again, we get
\begin{equation*}
\begin{split}
I^{(2)}_j&\le C\|f\|_{L^{p,\kappa}(w)}\cdot2^{{3jn}/2}\sum_{k=1}^\infty\frac{w(B)^{(1-\kappa)/p}}{w(2^{k+1}B)^{(1-\kappa)/p}}\\
&\le C\|f\|_{L^{p,\kappa}(w)}\cdot2^{{3jn}/2}.
\end{split}
\end{equation*}
Therefore
\begin{equation*}
\begin{split}
&\frac{1}{w(B)^{\kappa/p}}\bigg(\int_B\big|g^*_{\lambda,\alpha}(f)(x)\big|^pw(x)\,dx\bigg)^{1/p}\\
\le&\, C\|f\|_{L^{p,\kappa}(w)}\left(1+\sum_{j=1}^\infty 2^{-j\lambda n/2}2^{{3jn}/2}+\sum_{j=1}^\infty 2^{-j\lambda n/2}2^{jnp/2}\right)\\
\le&\, C\|f\|_{L^{p,\kappa}(w)},
\end{split}
\end{equation*}
where the last two series are both convergent under our assumption $\lambda>\max\{p,3\}$. Hence, by taking the supremum over all balls $B\subseteq\mathbb R^n$, we conclude the proof of Theorem 1.3.
\end{proof}

Finally, we remark that by using the arguments as in the proof of Theorems 1.2 and 1.3, we can also show the conclusion of Theorem 1.4. The details are omitted here.

\end{document}